\documentclass[10pt, leqno, a4paper]{amsart}

\usepackage[latin1]{inputenc}
\usepackage{amsfonts}
\usepackage{amsmath}
\usepackage{amssymb}
\usepackage{amsthm}
\usepackage[T1]{fontenc}
\usepackage[dvips]{graphicx}
\usepackage{enumerate}
\usepackage{verbatim}

\usepackage[pdftex]{hyperref}
\usepackage[matrix, arrow]{xy}

\author{Axel St\"abler}

\title{An explicit computation of a family of trivialising \'etale covers}

\DeclareMathOperator{\Proj}{Proj}

\DeclareMathOperator{\Syz}{Syz}

\DeclareMathOperator{\GL}{GL}

\def\cocoa{{\hbox{\rm C\kern-.13em o\kern-.07em C\kern-.13em o\kern-.15em
A}}}
\input xy
\xyoption{all}
\SelectTips{cm}{10}

\DeclareMathOperator{\Spec}{Spec}

\numberwithin{equation}{section}

\begin{document}
\swapnumbers
\theoremstyle{definition}
\newtheorem{Le}{Lemma}[section]
\newtheorem{Def}[Le]{Definition}
\newtheorem*{DefB}{Definition}
\newtheorem{Bem}[Le]{Remark}
\newtheorem{Ko}[Le]{Corollary}
\newtheorem{Theo}[Le]{Theorem}
\newtheorem*{TheoB}{Theorem}
\newtheorem{Bsp}[Le]{Example}
\newtheorem{Be}[Le]{Observation}
\newtheorem{Prop}[Le]{Proposition}
\newtheorem{Sit}[Le]{Situation}
\newtheorem{Que}[Le]{Question}
\newtheorem*{Con}{Conjecture}
\newtheorem{Dis}[Le]{Discussion}
\newtheorem{Prob}[Le]{Problem}
\newtheorem{Konv}[Le]{Convention}
\def\cocoa{{\hbox{\rm C\kern-.13em o\kern-.07em C\kern-.13em o\kern-.15em
A}}}
\address{Universit\"at Osnabr\"uck, Fachbereich 6: Mathematik/Informatik\\
Albrechtstr. 28a\\
49069 Osnabr\"uck\\
Germany}
\email{staebler@uni-mainz.de}

\curraddr{Johannes Gutenberg-Universi\"at Mainz\\ Fachbereich 08\\
Staudingerweg 9\\
55128 Mainz\\
Germany}

\begin{abstract}
We explicitly compute \'etale covers of the smooth Fermat curves $Y_{p+1} = \Proj k[u,v,w]/(u^{p+1} + v^{p+1} - w^{p+1})$ which trivialise the vector bundles $\Syz(u^2, v^2, w^2)(3)$, where $k$ is a field of characteristic  $p \geq 3$.
\end{abstract}
\maketitle
\allowdisplaybreaks[1]

\section*{Introduction}
In this paper we explicitly compute a trivialising \'etale cover $\varphi:X_{p +1} \to Y_{p+1}$ for the vector bundle $\mathcal{S} = \Syz(u^2, v^2, w^2)(3)$ on the Fermat curve $Y_{p+1}$ given by the equation $u^{p+1} + v^{p+1} - w^{p+1}$ over a field of positive characteristic $p \geq 3$. Such a cover corresponds to a representation of the \'etale fundamental group of $Y_{p+1}$.

A \emph{Frobenius periodicity} of a vector bundle $\mathcal{S}$ is an isomorphism $F^{e^\ast} \mathcal{S} \to \mathcal{S}$ for some $e \geq 1$, where $F$ denotes the absolute Frobenius morphism. By a classical result of Lange and Stuhler \cite[Satz 1.4]{langestuhler}, a vector bundle admitting such a Frobenius periodicity is \'etale trivialisable (the converse does not hold in general though -- see \cite[example below Theorem 1.1]{biswasducrohetlangestuhler} or \cite[Example 2.10]{brennerkaiddeepfrobeniusdescent}). The proof of Lange and Stuhler yields explicit local equations with gluing data for an \'etale cover.

Brenner and Kaid showed in \cite[Example 5.1]{brennerkaidfrobeniusperiodexplicit} that the bundle $\mathcal{S}$ admits a Frobenius periodicity with $e = 1$. We provide an explicit description of this isomorphism in terms of generators of the syzygy bundle. Then we compute the section ring of $X_{p+1}$ with respect to $\varphi^\ast \mathcal{O}_{Y_{p+1}}(1)$ and show that the covering obtained via the construction outlined in \cite[proof of Satz 1.4]{langestuhler} is not geometrically connected. We also compute the genera and the degrees of the connected components. This is a partial answer to a question posed by Brenner and Kaid (see \cite[Remark 5.2]{brennerkaidfrobeniusperiodexplicit}).

I thank Holger Brenner for useful discussions and for pointing out several simplifications to arguments in an earlier version of this article. Furthermore, the computer algebra system \cocoa\ (\cite{CocoaSystem}) has been helpful in the preparation of this paper.

\section{The isomorphism inducing Frobenius peridocity}

Throughout this section we denote $\Proj k[x,y,z]/(x^d + y^d - z^d)$ by $C$, where $k$ is a field of prime characteristic $p = 2d-1$. 
The goal of this section is to explicitly identify the isomorphism $F^\ast \Syz(x,y,z) \cong \Syz(x,y,z)(- \frac{3(p-1)}{2})$ described in \cite[Theorem 3.4]{brennerkaidfrobeniusperiodexplicit}, where $F: C \to C$ is the (absolute) Frobenius morphism. From this isomorphism we will then obtain the desired $F$-periodicity by passing to a suitable Fermat cover. For a locally free sheaf $\mathcal{S}$ we call a global section of $\mathcal{S}(m)$ a section of \emph{total degree} $m$.
 
\begin{Le}
The locally free sheaf $\mathcal{S}_1 = \Syz(x^p, y^p, x^{\frac{p+1}{2}} + y^{\frac{p+1}{2}})$ on $C$ is generated by the relations \[  R_0 = (y^{\frac{p-1}{2}}, x^{\frac{p-1}{2}}, - (xy)^{\frac{p-1}{2}}) \text{ and } R_1 = (-x, y, x^{\frac{p+1}{2}} - y^{\frac{p+1}{2}})\] in degree $p+1$ and $\frac{3p -1}{2}$ respectively. And we have an isomorphism $\mathcal{O}_C(-p -1) \oplus \mathcal{O}_C(-\frac{3p -1}{2}) \xrightarrow{R_1, R_0} \mathcal{S}_1$.
\end{Le}
\begin{proof}
One easily verifies that $R_0$ and $R_1$ are indeed syzygies of $\mathcal{S}_1$. 
Moreover, $\mathcal{S}_1$ is just the pull back of $\Syz_{\mathbb{P}^1}(x^p, y^p,  x^{\frac{p+1}{2}} + y^{\frac{p+1}{2}})$ along the Noether normalisation $\pi: C \to \mathbb{P}^1 = \Proj k[x,y]$ -- cf. \cite[Remark 2.3]{brennerkaidfrobeniusperiodexplicit}.

Note that $R_0$ and $R_1$ are linearly independent over $k[x,y]$. Hence, we obtain an exact sequence $0 \to \mathcal{O}_C(-p -1) \oplus \mathcal{O}_C(-\frac{3p -1}{2}) \to \mathcal{S}_1 \to \mathcal{T} \to 0$, where $\mathcal{T}$ is a torsion sheaf. Taking cohomology we obtain that $H^0(C, \mathcal{T})$ is zero ($H^1(C, \mathcal{T})$ vanishes since $\mathcal{T}$ has support on a closed affine subscheme). Hence, we have the desired isomorphism. 
\end{proof}

\begin{Le}
 The locally free sheaf $\mathcal{S}_2 = \Syz(x^p, y^p, (x^{\frac{p+1}{2}} + y^{\frac{p+1}{2}})^2)$ on $C$ is generated by the relations \[R_2 = (xy^{\frac{p-1}{2}}, 2x^{\frac{p+1}{2}} + y^{\frac{p+1}{2}}, -y^{\frac{p-1}{2}}) \text{ and } R_3 = (x^{\frac{p+1}{2}} + 2y^{\frac{p+1}{2}}, x^{\frac{p-1}{2}}y, -x^{\frac{p-1}{2}})\] in degree $\frac{3p +1}{2}$. And we have an isomorphism $ \mathcal{O}_C(- \frac{3p +1}{2})^2 \xrightarrow{R_2, R_3} \mathcal{S}_2$.
\end{Le}
\begin{proof}
The argument is similar to the previous lemma.
\end{proof}

\begin{Le}
\label{Lekernel}
The kernel of the surjective morphism \[\begin{xy}\xymatrix{\mathcal{O}_C^3 \ar[r]& \,\mathcal{S}_1(p+1) \oplus \mathcal{S}_2(\frac{3p +1}{2}) \ar[r]^{\varphi}& \Syz(x^p, y^p, z^p)(\frac{3p+1}{2})}\end{xy},\] where the first morphism is given by mapping $e_i$ to $R_i$ ($i =1,2,3$) and $\varphi$ is given by \[(f_1, f_2, f_3), (g_1, g_2, g_3) \longmapsto (z^{d-1}f_1 + g_1, z^{d-1} f_2 + g_2, f_3 + z g_3),\] is generated by the single relation $(z, -y, x)$.
\end{Le}
\begin{proof}
By \cite[Steps 3 and 4 of Theorem 3.4]{brennerkaidfrobeniusperiodexplicit} we have that the morphism is surjective and that its kernel is isomorphic to $\mathcal{O}_C(-1)$. A straightforward calculation shows that $(z, -y, x)$ is mapped to zero along this map.
\end{proof}

By abuse of notation we will denote this morphism by $\varphi$. With this notation we have\begin{align*}
\varphi(e_1) &= (-z^{d-1}x, yz^{d-1}, x^d -y^d),\\
\varphi(e_2) &= (xy^{d-1}, 2x^d + y^d, -y^{d-1}z),\\ 
\varphi(e_3) &= (x^d + 2y^d, x^{d-1}y, -x^{d-1}z).
\end{align*}

We thus obtain an exact sequence \[\begin{xy} \xymatrix{ 0 \ar[r] & \mathcal{O}_C(-1) \ar[r]^<<<<<{(z,-y,x)} & \mathcal{O}_C^3 \ar[r]^>>>>>\varphi & \Syz(x^p, y^p, z^p)(\frac{3p +1}{2}) \ar[r] & 0.} \end{xy}
\]

\begin{Prop}
\label{FPullbackIsom}
The morphism $\alpha: \Syz(x^p, y^p, z^p)(\frac{3p +1}{2}) \to \Syz(x, y, z)(2)$ given on generators by 
\begin{align*}
&(-z^{d-1}x, yz^{d-1}, x^d - y^d) \longmapsto (-y, x, 0),\\
&(xy^{d-1}, 2x^d + y^d, -y^{d-1}z) \longmapsto (-z,0,x),\\
&(x^{d} + 2y^d, x^{d-1}y, -x^{d-1}z) \longmapsto (0,-z,y)
\end{align*} is an isomorphism.
\end{Prop}
\begin{proof}
The Koszul complex of $\Syz(z, -y, x)$ yields the short exact sequence
\[\begin{xy} \xymatrix{ 0 \ar[r] & \mathcal{O}_C(-1) \ar[r]^<<<<<{(z,-y,x)} & \mathcal{O}^3_C \ar[r]^>>>>>\psi & \Syz(z, -y, x)(2) \ar[r] & 0,} \end{xy}\]
where $\psi$ is given by
\begin{align*}
e_1 &\longmapsto (x,0,-z),\\
e_2 &\longmapsto (y,z,0),\\
e_3 &\longmapsto (0,x,y).
\end{align*}
Mapping $(a_1, a_2, a_3)$ to $(a_3, -a_2, a_1)$ yields an isomorphism $\Syz(z,-y,x)(2) \to \Syz(x,y,z)(2)$. Together with the observation after Lemma \ref{Lekernel} this yields the desired isomorphism $\Syz(x^p, y^p, z^p)(\frac{3p + 1}{2}) \to \Syz(x,y,z)(2)$.
\end{proof}

\begin{Ko}
\label{KoFrobPeriod}
Let $k$ be a field of characteristic $p \geq 3$ and let \[Y_{p+1} = \Proj k[u,v,w]/(u^{p+1} + v^{p+1} - w^{p+1}).\] Then we have the Frobenius periodicity $F^\ast \Syz(u^2, v^2, w^2)(3) \to \Syz(u^2, v^2, w^2)(3)$ given on generators by
\begin{align*}
 &(-w^{p-1}u^2, v^2w^{p-1}, u^{p+1} - v^{p+1}) \longmapsto (-v^2, u^2, 0),\\
 &(u^2v^{p-1}, 2u^{p+1} + v^{p+1}, -v^{p-1}w^2) \longmapsto (-w^2,0,u^2),\\
 &(u^{p+1} + 2v^{p+1}, u^{p-1}v^2, -u^{p-1}w^2) \longmapsto (0,-w^2,v^2).
\end{align*}
\end{Ko}
\begin{proof}
Set $d = \frac{p+1}{2}$. Then the isomorphism is induced by the finite cover $C^{2d} \to C$ given by $x \mapsto u^2, y \mapsto v^2, z \mapsto w^2$ (cf.\ \cite[Example 5.1]{brennerkaidfrobeniusperiodexplicit}) and the isomorphism described in Proposition \ref{FPullbackIsom}.
\end{proof}

\begin{Bem}
Brenner and Kaid actually established a Frobenius periodicity for $\Syz(u^2,v^2,w^2)(3)$ on curves $C = \Proj k[u,v,w]/(u^{2d} + v^{2d} - w^{2d})$, where $k$ is a field of characteristic $p \equiv -1 \mod 2d$. Let us write $p = d(l+1) -1$ with $l$ odd.
It would be even more interesting to have explicit descriptions of the isomorphisms $F^\ast \Syz(u^2, v^2, w^2)(3) \to \Syz(u^2, v^2, w^2)(3)$ in the case where $d$ is fixed and $l$ varies. For then one would have a relative curve over $\Spec \mathbb{Z}$ which might be interesting with respect to the Grothendieck-Katz $p$-curvature conjecture -- cf. \cite[Remark 5.3]{brennerkaidfrobeniusperiodexplicit}.
\end{Bem}

\section{Some computations}
\label{somecomputations}
Throughout this section $k$ is a field of characteristic $p = 2d -1$ and $Y = \Proj k[u,v,w]/(u^{2d} + v^{2d} - w^{2d})$. We denote the Frobenius periodicity described in Corollary \ref{KoFrobPeriod} by $\alpha: F^\ast \Syz(u^2, v^2, w^2)(3) \to \Syz(u^2, v^2, w^2)(3)$.
We write $\mathcal{S} = \Syz(u^2, v^2, w^2)(3)$. The locally free sheaf $\mathcal{S}$ is generated by 
\begin{align*}
s_1 &= (-v^2, u^2, 0),\\
s_2 &= (-w^2,0,u^2),\\
s_3 &= (0,-w^2,v^2) 
\end{align*}
 in total degree $1$. Furthermore, we write $\mathcal{S}' = F^\ast \mathcal{S} = \Syz(u^{2p}, v^{2p}, w^{2p})(3p)$, for which we fix generators \begin{align*} s_1' &= (-w^{2d-2}u^2, v^2w^{2d-2}, u^{2d} - v^{2d}),\\
 s_2' &= (u^2v^{2d-2}, 2u^{2d} + v^{2d}, -v^{2d-2}w^2),\\
 s_3' &= (u^{2d} + 2v^{2d}, u^{2d-2}v^2, -u^{2d-2}w^2)
\end{align*} in total degree $1$.

The goal of this section is to collect the data that are needed for construction of the \'etale cover as outlined in \cite[Satz 1.4]{langestuhler}. For the convenience of the reader we shall review this (with notations tailored to our situation).

Assume that we have a smooth curve $Y$ and a locally free sheaf $\mathcal{S}$ on $Y$ with Frobenius periodicity $\alpha: F^{\ast} \mathcal{S} \to \mathcal{S}$. Let $U_1,U_2$ be a trivialising open cover for $\mathcal{S}$. Let $\psi_{U_1}: \mathcal{S}\vert_{U_1} \to \mathcal{O}_{U_1}^2$ and $\psi_{U_2}: \mathcal{S}\vert_{U_2} \to \mathcal{O}_{U_2}^2$ be the transition mappings and denote $\psi_{U_1} \psi_{U_2}^{-1} \in \GL_2(\mathcal{O}^2_{U_1 \cap U_2})$ by $T$. For a matrix $A = (a_{ij})$ with coefficients in a ring denote by $A^{(p)}$ the matrix whose entries are the $a_{ij}^p$. We denote by $H_{U_1}$ the mapping that makes the following diagram commutative
\[
\begin{xy}
 \xymatrix{ \mathcal{S}\vert_{U_1} \ar[r]^{\alpha^{-1}\vert_{U_1}} \ar[d]^{\psi_{U_1}} & \mathcal{S}'\vert_{U_1} \ar[d]^{\psi^{(p)}_{U_1}} \\
 \mathcal{O}^2_{U_1} \ar[r]^{H_{U_1}}& \mathcal{O}^2_{U_1}}
\end{xy}
\] and similarly for $H_{U_2}$.

Let now \[A = \begin{pmatrix} a & b \\ c & d \end{pmatrix} \text { and } B = \begin{pmatrix} \alpha & \beta \\ \gamma & \delta \end{pmatrix},
 \] where the entries are indeterminates.
With this notation one has by \cite[proof of Satz 1.4]{langestuhler} that a trivialising \'etale cover $g:X \to Y$ for $\mathcal{S}$ on $Y$ is obtained by gluing the algebras \[A_{U_1} = \mathcal{O}_Y(U_1)[a, b, c, d, (\det A)^{-1}]/((A^{(p)}A^{-1} - H_{U_1})_{i,j} \, \vert \, i,j =1,2)\] and \[ B_{U_2} = \mathcal{O}_Y(U_2)[\alpha, \beta, \gamma, \delta, (\det B)^{-1}]/((B^{(p)}B^{-1} - H_{U_2})_{i,j} \, \vert \, i,j =1,2)\] along the the identifications $(T B)_{ij} = (A)_{ij}$ on $U_1 \cap U_2$. The morphism $g: X \to Y$ is induced by the inclusions $\mathcal{O}_Y(U_1) \to A_{U_1}$ and $\mathcal{O}_Y(U_2) \to B_{U_2}$.

\begin{Le}
\label{lemmatransitionfunctions}
Let $Y = \Proj k[u,v,w]/(u^{2d} + v^{2d} - w^{2d})$ and $\mathcal{S} = \Syz(u^2, v^2, w^2)(3)$. Then $U = D_+(u), W = D_+(w)$ is a trivialising cover for $\mathcal{S}$ with respect to the isomorphisms \[\xymatrix@1{\psi_U: \mathcal{S}\vert_U \ar[r]& \mathcal{O}^2_U},\quad \frac{s_1}{u} \longmapsto e_1, \frac{s_2}{u} \longmapsto e_2, \] \[\xymatrix@1{\psi_W: \mathcal{S}\vert_W \ar[r]& \mathcal{O}^2_W},\quad \frac{s_2}{w} \longmapsto e_1, \frac{s_3}{w} \longmapsto e_2.\] The transition mapping for the cover $U, W$ is given by \[\psi_U \psi_W^{-1}= \begin{pmatrix} 0 & -\frac{w}{u}\\ \frac{u}{w} & \frac{v^2}{uw} \end{pmatrix} \in GL_2(\mathcal{O}_Y(U \cap W)). \] 
\end{Le}
\begin{proof}
First of all, note that we have the relation $R: w^2s_1 - v^2s_2 + u^2s_3 = 0$ among the fixed generators of $\mathcal{S}$. 
On $U$ this relation may be rewritten as $\frac{w^2}{u^2} \frac{s_1}{u} + \frac{s_3}{u} - \frac{v^2}{u^2} \frac{s_2}{u} = 0$ so that $\frac{s_1}{u}, \frac{s_2}{u}$ generate $\mathcal{S}\vert_U$.
 Consequently, $\mathcal{S}\vert_U$ is free. Indeed, $\mathcal{S}$ is of rank $2$ and we have the exact sequence $\mathcal{O}_U^2 \to \mathcal{S}\vert_{U} \to 0$ (the $\mathcal{O}_Y(a)$ are invertible -- cf. \cite[Proposition II.5.12 (a)]{hartshornealgebraic}). Hence, the kernel is of rank zero, thus zero since $\mathcal{O}_U$ is torsion-free.

One sees similarly that $\mathcal{S}\vert_W$ is free, generated by $\frac{s_2}{w}, \frac{s_3}{w}$.

We thus have isomorphisms
\begin{align*}
&\xymatrix@1{\psi_U: \mathcal{S}\vert_{U \cap W} \ar[r]& \mathcal{O}_{U \cap W}^2}, \frac{s_1 }{u } \longmapsto e_1, \frac{s_2 }{u} \longmapsto e_2,\\
&\xymatrix@1{\psi_V: \mathcal{S}\vert_{U \cap W} \ar[r]& \mathcal{O}_{U \cap W}^2}, \frac{s_2}{w} \longmapsto e_1, \frac{s_3}{w} \longmapsto e_2.
\end{align*} And we obtain $\frac{s_2}{w} = \frac{u}{w} \frac{s_2}{u}$ and $\frac{s_3}{w} = \frac{v^2}{uw} \frac{s_2}{u} - \frac{w}{u} \frac{s_1}{u}$ using $R$ restricted to $U \cap W$. Whence the transition matrix.
\end{proof}

\begin{Le}
Let $Y = \Proj k[u,v,w]/(u^{2d} + v^{2d} - w^{2d})$, where $k$ is a field of characteristic $p = 2d -1$ and let $\mathcal{S}' = F^\ast  (\Syz(u^2, v^2, w^2)(3)) = \Syz(u^{2p}, v^{2p}, w^{2p})(3p)$. Then $\frac{s_1'}{u}, \frac{s_2'}{u} $ is a basis of $\mathcal{S}'\vert_U$ and $\frac{s'_2}{w}, \frac{s'_3}{w}$ is a basis of $\mathcal{S}'\vert_W$.
\end{Le}
\begin{proof}
Note that we have a relation $w^2 s_1' - v^2s_2' + u^2s_3' = 0$. The syzygies $s_i'$ are of total degree $2d + 2p - 3p = 1$ in $\mathcal{S}'$. Since $\mathcal{S}' = F^\ast \mathcal{S}$ we have by Lemma \ref{lemmatransitionfunctions} that $\mathcal{S}'\vert_U, \mathcal{S}'\vert_W$ are free and similarly to the proof of Lemma \ref{lemmatransitionfunctions} one sees that any two generators have to be free.

On $U$ this relation may be written as $\frac{w^2}{u^2} \frac{s_1'}{u} - \frac{v^2}{u^2} \frac{s_2'}{u} + \frac{s_3'}{u}= 0$ so that $\frac{s_1'}{u}, \frac{s_2'}{u}$ generate $\mathcal{S}'\vert_U$. 
Similarly we have on $W$ that $\frac{s_2'}{w}, \frac{s_3'}{w}$ generate $\mathcal{S}'_W$.
\end{proof}

\begin{Le}
\label{LBaseChange}
The change of basis matrix from $B'_U := \{\frac{s_1'}{u}, \frac{s_2'}{u}\}$ to $B^{(p)}_U := \{ \frac{s_1^p}{u^p}, \frac{s_2^p}{u^p} \}$ is given by \[\begin{pmatrix} \frac{v^2w^{p-1}}{u^{p+1}} & 2 + \frac{v^{p+1}}{u^{p+1}}\\ 1 - \frac{v^{p+1}}{u^{p+1}} & -\frac{v^{p-1}w^2}{u^{p+1}}\end{pmatrix}.\]
And the change of basis matrix from $B'_W := \{\frac{s_2'}{w}, \frac{s_3'}{w}\}$ to $B^{(p)}_W := \{ \frac{s_2^p}{w^p}, \frac{s_3^p}{w^p} \}$ is given by \[\begin{pmatrix}-\frac{u^2 v^{p-1}}{w^{p+1}} & -\frac{u^{p+1} + 2v^{p+1}}{w^{p+1}} \\ - \frac{2u^{p+1} + v^{p+1}}{w^{p+1}} & - \frac{u^{p-1}v^2}{w^{p+1}}\end{pmatrix}. \]
\end{Le}
\begin{proof}
We first compute the change of basis matrix on $U$.
Write $\alpha \frac{s_1^{p}}{u^p} + \beta \frac{s_2^{p}}{u^p} = \frac{s_1'}{u}$. Looking at the second component shows that $\alpha = \frac{v^2w^{p-1}}{u^{p+1}}$ and restricting to the third componend yields $\beta = (1 - \frac{v^{p+1}}{u^{p+1}})$. Similarly one obtains that $(2 + \frac{v^{p+1}}{u^{p+1}}) \frac{s_1^{p}}{u^p} - \frac{v^{p-1}w^2}{u^{p+1}} \frac{s_2^p}{u^p} =\frac{s_2'}{u}$.

On $W$ we have \[-\frac{u^2 v^{p-1}}{w^{p+1}} \frac{s_2^p}{w^p} - \frac{2u^{p+1} + v^{p+1}}{w^{p+1}} \frac{s_3^p}{w^p}= \frac{s_2'}{w}\] and \[-\frac{u^{p+1} + 2v^{p+1}}{w^{p+1}} \frac{s_2^p}{w^p} - \frac{u^{p-1}v^2}{w^{p+1}}\frac{s_3^p}{w^p} = \frac{s_3'}{w}.\]
\end{proof}

\begin{Prop}
\label{PLocalFIsoms}
The isomorphisms $H_U: \mathcal{O}_U^2 \to \mathcal{O}_U^2 $ and $H_W: \mathcal{O}_W^2 \to \mathcal{O}_W^2$ are given by \[H_ U = \begin{pmatrix} \frac{v^2w^{p-1}}{u^{p+1}} & 2 + \frac{v^{p+1}}{u^{p+1}}\\ 1 - \frac{v^{p+1}}{u^{p+1}} & -\frac{v^{p-1}w^2}{u^{p+1}}\end{pmatrix} \text{ and } H_W = \begin{pmatrix}-\frac{u^2 v^{p-1}}{w^{p+1}} & -\frac{u^{p+1} + 2v^{p+1}}{w^{p+1}} \\ - \frac{2u^{p+1} + v^{p+1}}{w^{p+1}} & - \frac{u^{p-1}v^2}{w^{p+1}}\end{pmatrix} \] with respect to standard bases. 
\end{Prop}
\begin{proof}
In order to obtain $H_U$, note that $H_U = \psi_U^{(p)} \alpha^{-1}\vert_U \psi^{-1}_U$. Moreover, $\psi_U^{-1}$ is the identity matrix with respect to the bases $B_U := \{\frac{s_1}{u}, \frac{s_2}{u}\}$ and standard basis. Likewise, $\alpha^{-1}\vert_U$ is the identity matrix with respect to the bases $B_U$ and $B'_U := \{\frac{s_1'}{u}, \frac{s_2'}{u}\}$, as is $\psi^{(p)}_U$ with respect to $B^{p}_U :=  \{\frac{s_1^p}{u^p}, \frac{s_2^p}{u^p}\}$ and standard basis. So $H_U$ with respect to standard basis is none other than the matrix computed in Lemma \ref{LBaseChange}.

The claim about $H_W$ follows similarly.
\end{proof}

\section{The example}
In this section we compute the section ring of the \'etale cover \[\begin{xy}\xymatrix{g:X \ar[r]& Y = \Proj k[u,v,w]/(u^{p+1} + v^{p+1} - w^{p+1})}\end{xy}\] that is obtained from the Frobenius periodicity \[F^\ast (\Syz(u^2, v^2, w^2)(3)) \cong \Syz(u^2,v^2,w^2)(3)\] via \cite[Satz 1.4]{langestuhler}. We also show that $X$ is not geometrically connected and we compute the genera and the degrees of its irreducible components when $k$ contains a $(p-1)$th root of $-2$. In fact, in this case any two irreducible components of $X$ are isomorphic. In the following we will denote the ring $k[u,v,w]/(u^{p+1} + v^{p+1} - w^{p+1})$ by $R$. Note that this is the section ring induced by $\mathcal{O}_Y(1)$ on $Y$.

Following the construction outlined at the beginning of section \ref{somecomputations} and using Lemma \ref{lemmatransitionfunctions} we obtain that \begin{align}
a &= -\frac{w}{u} \gamma&  c &= \frac{u}{w}\alpha + \frac{v^2}{uw} \gamma \\
b &= - \frac{w}{u} \delta& d &= \frac{u}{w} \beta  + \frac{v^2}{uw} \delta. 
\end{align}

Explicitly, one has \begin{equation}\begin{split}A_U &= \mathcal{O}_Y(U)[a,b,c,d, \det A^{-1}]/(\det A^{-1} (a^pd - cb^p) - \frac{v^2}{u^2} \frac{w^{p-1}}{u^{p-1}},\\&\det A^{-1}(b^pa - a^pb) - (2 + \frac{v^{p+1}}{u^{p+1}}),
\det A^{-1} (c^pd -cd^p) - (1 - \frac{v^{p+1}}{u^{p+1}}),\\&
\det A^{-1} (d^pa -bc^p) + \frac{w^2}{u^2} \frac{v^{p-1}}{u^{p-1}}).\end{split}\end{equation}

Our first goal is to compute the section ring $S = \bigoplus_{n \geq 0} H^0(X, g^\ast \mathcal{O}_Y(n))$. Note that we have a finite ring extension $R = k[u,v,w]/(u^{p+1} + v^{p+1} - w^{p+1}) \to S$ (cf. \cite[Lemma 3.5]{brennerstaeblerdaggersolid} -- the injectivity still holds since the morphisms on the stalks are all injective).

\begin{Le}
The natural maps $\begin{xy}\xymatrix{R \ar[r]& A_U[u,u^{-1}]}\end{xy}$ and \[\begin{xy} \xymatrix{A_U \ar[r]& A_U[u,u^{-1}] \ar[r]& A_U[u, u^{-1}, w^{-1}]}\end{xy}\] are all injective.
\end{Le}
\begin{proof}
It is enough to show that $u, w$ are non-zero divisors in the section ring $S$ induced by $g^\ast \mathcal{O}_Y(1)$. Indeed, $A_U[u,u^{-1}]$ is isomorphic to $S_u$ (see e.\,g.\ \cite[2.2.1]{EGAII}). In particular, the map $A_U \to A_U[u,u^{-1}]$ is just the natural inclusion.

Let $X_i$ be a connected component of $X$, we then obtain a finite morphism $X_i \to Y$ which is dominant. The pull back of $\mathcal{O}_Y(1)$ to $X_i$ induces a section ring $S_i$ and $S$ is isomorphic to $S_1 \times \ldots \times S_n$. As $X_i \to Y$ is dominant $R \to S_i$ is injective and $S_i$ is integral since $X_i$ is irreducible. So if $u$ were a zero divisor it would map to $0$ in some $S_i$. But this is impossible since the morphism $R \to S$ is injective.
\end{proof}

\begin{Prop}
\label{PSectionRing}
The section ring induced by $g^\ast \mathcal{O}_Y(1)$ on $X$ is given by \[S = R[w \alpha, w \beta ,w \gamma, w \delta, \alpha \frac{u^2}{w} + \gamma \frac{v^2}{w}, \beta \frac{u^2}{w} + \delta \frac{v^2}{w}, (\alpha \delta -\beta \gamma)^{-1}]\] viewed as a subring of $A_U[u, u^{-1}, w^{-1}]$ with the identifications of $(3.1)$ and $(3.2)$. 
\end{Prop}
\begin{proof}
Denote the section ring by $T$. Then one has $T_u = A_U[u, u^{-1}]$ and $T_w = B_W[w,w^{-1}]$ (see e.\,g.\ \cite[2.2.1]{EGAII}). Since $D_+(u), D_+(w)$ form a covering of $Z$ it is enough to show that $T_u = S_u$ and $T_w = S_w$.

Using $(3.1)$ and $(3.2)$ we see that $A_U[u, u^{-1}]$ is given by \[\mathcal{O}_Y(U)[w \gamma, w \delta, \frac{u^2}{w} \alpha + \frac{v^2}{w} \gamma, \frac{u^2}{w} \beta + \frac{v^2}{w} \delta, (ad - bc)^{-1}, u, u^{-1}].\] One readily computes that $ad - bc = \alpha \delta - \beta \gamma$. So it is clear that $A_U[u, u^{-1}] \subseteq S_u$. Moreover, $((\frac{u^2}{w} \alpha + \frac{v^2}{w} \gamma) w^2 - v^2 w \gamma) u^{-2} = w \alpha$ and $(\frac{u^2}{w} \beta + \frac{v^2}{w} \delta) w^2 - u^2w \delta) u^{-2} = w \beta$. Hence, both $w \alpha$ and $w \beta$ are contained in $A_U[u,u^{-1}]$. This shows the converse inclusion.

The equality $S_w = B_W[w, w^{-1}]$ is immediate.
\end{proof}

\begin{Le}
\label{LNotConnected}
We have $ad - bc \notin k$ in $S$. Moreover, $(ad - bc)^{p-1} = -2$.
\end{Le}
\begin{proof}
Assume to the contrary that $D = ad -bc$ is contained in $k$. Then necessarily $D \in k^{\times}$. In particular, we may replace $A_U$ with the isomorphic ring \begin{equation*}\begin{split} A'_U &= \mathcal{O}_Y(U)[a,b,c,d]/(a^pd - cb^p - D \frac{v^2}{u^2} \frac{w^{p-1}}{u^{p-1}},b^pa - a^pb - D(2 + \frac{v^{p+1}}{u^{p+1}}),\\&
c^pd -cd^p - D(1 - \frac{v^{p+1}}{u^{p+1}}),
d^pa -bc^p + D \frac{w^2}{u^2} \frac{v^{p-1}}{u^{p-1}}) \end{split}\end{equation*} and consider $S$ as a subring of this quotient. 

We invert $u$ in $S$ and then kill $w$, i.\,e.\ we look at the ring $T := A'_U[u^{-1}]/(w)$. Thinking of this ring as a quotient of $R[u^{-1}]/(w)[a,b,c,d]$ we can rewrite the defining ideal as follows (note that we have $u^{p+1} + v^{p+1} = 0$),
\[
(a^pd - cb^p, b^pa - a^pb - (ad -bc), c^pd - cd^p -2( ad - bc), d^pa - bc^p).
\]
This ring is not the zero ring and we still must have $D \in k^{\times}$ in $T$. But this is impossible since the defining ideal is contained in $(a,b,c,d)$.

Since $\det A^{(p)} = (\det A)^p$ and by the multiplicative property of the determinant we obtain $\det A^{p-1} = (ad - bc)^{p-1} = \det H_U = -2$ in $A_U$.
\end{proof}

\begin{Ko}
The \'etale cover $X \to Y$ obtained via the Frobenius periodicity $F^\ast \mathcal{S} \to \mathcal{S}$ is not geometrically connected. 
\end{Ko}
\begin{proof}
As $X$ is \'etale the section ring is a direct product of normal domains. Assuming that $k$ is algebraically closed, the section ring is a domain if and only if $H^0(X, \mathcal{O}_X) = k$. As we have seen $ad - bc \in H^0(X, \mathcal{O}_X)$. So if $X$ were connected then $ad -bc \in k$. But this is not the case by Lemma \ref{LNotConnected}.
\end{proof}

\begin{Prop}
\label{PNumberIrrComp}
Assume that $k$ contains a $(p-1)$th root of $-2$ which we call $\eta$. Then $X$ has exactly $p- 1$ connected components which are all isomorphic. The connected components $X_i$ are isomorphic to $\Proj S/(ad - bc + \zeta^i \eta)$, where $\zeta$ is a primitve $(p-1)$th root of unity\footnote{In other words, any generator of $\mathbb{F}_p^\times$.} and $S$ denotes the section ring associated to $g^\ast \mathcal{O}_Y(1)$ (i.\,e.\ the one described in Proposition \ref{PSectionRing}).
\end{Prop}
\begin{proof}
We claim that the irreducible components of $X$ are the $X_i = \Proj S/(ad - bc + \zeta^i \eta)$ for $i = 1, \ldots, p-1$. At least one of the $X_i$ is non-empty, for otherwise all the $ad - bc +  \zeta^i \eta$ were units and thus their product $(ad - bc)^{p-1} + 2$ would be nonzero -- contradicting Lemma \ref{LNotConnected}.

Next we will show that the rings $T_i = S/(ad - bc + \zeta^i\eta )$ are all isomorphic. So fix indices $i, j$. We obtain an automorphism of $A_U[w^{-1}, u, u^{-1}]$ by multiplying $a, c$ by $\zeta^{-j}$ and by multiplying $b ,d$ with $\zeta^{i}$. Since the ideal in $(3.3)$ is mapped to itself this is well-defined. Furthermore, this induces an automorphism of $S$. Finally, $(ad - bc + \zeta^i \eta)$ is mapped to $(ad - bc + \zeta^j \eta)$. Hence, this also induces an isomorphism of $T_i$ with $T_j$. It follows in particular that all the $X_i$ are non-empty.

 We still have to show that the $X_i$ are the irreducible components. Since $S$ is normal, we have $S = S_1 \times \ldots \times S_n$, where the $S_i$ are normal integral $k$-domains. As $(ad -bc)^{p-1} = -2$ we must have that $ad -bc \mapsto ( \zeta^{i_1} \eta, \zeta^{i_2} \eta, \ldots, \zeta^{i_n} \eta)$. Here we see that the $T_i$ are therefore of the form $S_{j} \times \ldots \times S_{j + k}$. In particular, the $T_i$ are again normal. As the zeroth graded component of $S$ is $k[ad - bc]$ we obtain that the zeroth graded component of $T_i$ is $k$. Hence, the $T_i$ are integral domains. We therefore also obtain $n = p-1$ and the $X_i$ have to be the irreducible components.
\end{proof}

\begin{Le}
\label{LMatrizenverschiebung}
Let $R$ be a commutative ring and let $A,B,C$ square matrices of dimension $n$ with entries in $R$. Assume furthermore that $\det B$ is invertible in $R$. Then the ideal generated by the entries of $G := (A B^{-1} - C)$ is equal to the ideal generated by the entries of $H := (A - CB)$  
\end{Le}
\begin{proof}
Write $G = (g_{ij})$ and similarly for $B, H$.
Multiplying $G$ with $B$ from the right we obtain $\sum_j g_{ij} b_{jl} = h_{il}$. This proves one inclusion. The other inclusion is obtained by multiplying $H$ from the right by $B^{-1} = (\det B)^{-1} B^{\#}$, where $B^{\#}$ denotes the adjoint matrix.
\end{proof}

\begin{Prop}
\label{PDegreeCovering}
Assume that assume that $k$ contains a $(p-1)$th root of $-2$ and let $X_i = \Proj S/(r_i)$ be an irreducible component of $X = \Proj S$. Then the degree of the induced morphism $g_i: X_i \to Y = \Proj k[u,v,w]/(u^{2d} + v^{2d} - w^{2d})$ is $p (p^2-1)$.
\end{Prop}
\begin{proof}
The degree of the morphism is the $k$-dimension of the global sections of a fiber. Fix the point $u, w =1$, $v=0$ on $Y$. Then the matrix $H_U$ in Proposition \ref{PLocalFIsoms} reduces to
\[
 \begin{pmatrix}
  0 & 2\\
1 & 0
 \end{pmatrix}.
\] So by Lemma \ref{LMatrizenverschiebung} we obtain that the fiber (of the whole covering) is given by modding out $(2c -a^p, 2d - b^p, a - c^p, b - d^p, (ad - bc)^{p-1} + 2)$ in $k[a,b,c,d]$.  Observe, that $c$ and $d$ are units in this quotient ring. Indeed, modding out $c$ we obtain $a =0$ from the third generator and therefore, using the last generator, $2 = 0$. Hence, $c$ had to be a unit. The proof for $d$ works similarly.
We may thus rewrite this ideal as $(c^{p^2 -1} - 2, d^{p^2 -1} - 2, (cd^p - c^pd)^{p-1} + 2)$.

In counting ($\bar{k}$-valued) points we see from the first two generators that points $(c,d)$ have to satisfy $d \mapsto \zeta c$ and $c^{p^2 -1} =2$, where $\zeta$ is a $(p^2-1)$th root of unity. Furthermore, we must have $\zeta^{p-1} \neq 1$. Otherwise we would obtain a contradiction from the last equation. We claim that all these remaining choices yield points.

Indeed, we then have \[\begin{split}(cd^p - c^p d)^{p-1} &= (c^{p+1} \zeta^p -c^{p+1} \zeta)^{p-1} = c^{p^2 -1} \zeta^{p-1} (\zeta^{p-1} - 1)^{p-1} \\&= 2 \zeta^{p-1} \frac{\zeta^{p(p-1)} -1}{\zeta^{p-1} -1} = 2 \frac{\zeta^{(p+1)(p-1)} - \zeta^{p-1}}{\zeta^{p-1} -1} = -2.\end{split}\]

So the dimension of a fiber of the whole covering is $(p^2 -1) (p^2 -1 - (p-1)) = (p^2 -1) p (p-1)$.
Since any two irreducible components are isomorphic by Proposition \ref{PNumberIrrComp} we obtain the dimension of the fiber over an irreducible component by dividing by the number of irreducible components, which is $p-1$ -- whence the claim. 
\end{proof}

Summing up what we have proved we obtain

\begin{Theo}
Let $k$ be a field of characteristic $p = 2d -1$ containing a $(p-1)$th root of $-2$, $R = k[u,v,w]/(u^{p+1} + v^{p+1} - w^{p+1})$, $Y = \Proj R$ and $\mathcal{S} = \Syz(u^2, v^2, w^2)(3)$. Then there is a Frobenius periodicity $F^\ast( \Syz(u^2, v^2, w^2)(3)) \cong \Syz(u^2, v^2, w^2)(3)$. The trivialising \'etale cover $g:X \to Y$ obtained via this periodicity along the construction outlined in \cite[Satz 1.4]{langestuhler} is given by \[\Proj R[w \alpha, w \beta ,w \gamma, w \delta, \alpha \frac{u^2}{w} + \gamma \frac{v^2}{w}, \beta \frac{u^2}{w} + \delta \frac{v^2}{w}, (\alpha\delta -\beta\gamma)^{-1}]\] viewed as a subring of $A_U[u, u^{-1}, w^{-1}]$, where \begin{equation*}\begin{split}A_U &= \mathcal{O}_Y(D_+(u))[a,b,c,d, \det A^{-1}]/(\det A^{-1} (a^pd - cb^p) - \frac{v^2}{u^2} \frac{w^{p-1}}{u^{p-1}},\\&\det A^{-1}(b^pa - a^pb) - (2 + \frac{v^{p+1}}{u^{p+1}}),
\det A^{-1} (c^pd -cd^p) -(1 - \frac{v^{p+1}}{u^{p+1}}),\\& \det A^{-1} (d^pa -bc^p) + \frac{w^2}{u^2} \frac{v^{p-1}}{u^{p-1}})\end{split}\end{equation*} and 
\begin{align*}
a &= -\frac{w}{u} \gamma&  c &= \frac{u}{w}\alpha + \frac{v^2}{uw} \gamma \\
b &= - \frac{w}{u} \delta& d &= \frac{u}{w} \beta  + \frac{v^2}{uw} \delta. 
\end{align*}
This cover consists of $p-1$ irreducible components $X_i$ which are all isomorphic. The genus of an irreducible component is $p(p^2-1)(\frac{p(p-1)}{2} -1) +1$ and the degree of $g_i = g\vert_{X_i}: X_i \to Y$ is $p (p^2-1)$.
\end{Theo}
\begin{proof}
The claim about the genus follows from Hurwitz' theorem (\cite[Corollary IV.2.4]{hartshornealgebraic}) since $g$ is unramified, because the genus of $Y$ is $\frac{p(p-1)}{2} $ and since $\deg g_i = p(p^2-1)$ by Proposition \ref{PDegreeCovering}.
\end{proof}

\begin{Bsp}
For $p =3$ the bundle $\Syz(u^2,u^2,u^2)(3)$ is already trivial on the Fermat quartic $Y = \Proj k[u,v,w]/(u^4 + v^4 - w^4)$ since it is the twisted pull back of a bundle over the quadric $u^2 + v^2 - w^2$. As the quadric is isomorphic to $\mathbb{P}^1_k$ the bundle splits.

For $p =5$ we consider $\mathcal{S} = \Syz(u^2, v^2, w^2)(3)$ on $Y = \Proj k[u,v,w]/(u^6 + v^6 - w^6)$, where $k$ is a field of characteristic $5$ containing a $4$th root of $-2$. Then $\mathcal{S}$ has no global sections and is thus not the trivial bundle. By Proposition \ref{PNumberIrrComp} the trivialising \'etale cover $X \to Y$ has $4$ connected components. Each component $X_1, X_2,X_3, X_4$ is a curve of genus $1261$ and the morphisms $X_i \to Y$ have degree $120$. 
\end{Bsp}

\bibliography{bibliothek.bib}

\providecommand{\bysame}{\leavevmode\hbox to3em{\hrulefill}\thinspace}
\providecommand{\MR}{\relax\ifhmode\unskip\space\fi MR }
\providecommand{\MRhref}[2]{%
  \href{http://www.ams.org/mathscinet-getitem?mr=#1}{#2}
}
\providecommand{\href}[2]{#2}
\begin{thebibliography}{1}

\bibitem{biswasducrohetlangestuhler}
I.~Biswas and L.~Ducrohet, \emph{An analog of a theorem of {L}ange and
  {S}tuhler for principal bundles}, C.R. Acad. Sci. Paris, Ser. I \textbf{345}
  (2007), no.~9, 495--497.

\bibitem{brennerkaiddeepfrobeniusdescent}
H.~Brenner and A.~Kaid, \emph{On deep {F}robenius descent and flat bundles},
  Math. Res. Lett. \textbf{15} (2008), no.~5-6, 1101--1115.

\bibitem{brennerkaidfrobeniusperiodexplicit}
\bysame, \emph{An explicit example of {F}robenius peridocity},  (2010).

\bibitem{brennerstaeblerdaggersolid}
H.~Brenner and A.~St\"abler, \emph{Dagger closure and solid closure in graded
  dimension two}, arXiv:1104.3748v1 (2011).

\bibitem{CocoaSystem}
{CoCoA}Team, \emph{{{\hbox{\rm C\kern-.13em o\kern-.07em C\kern-.13em
  o\kern-.15em A}}}: a system for doing {C}omputations in {C}ommutative
  {A}lgebra}, Available at \/ {\tt http://cocoa.dima.unige.it}.

\bibitem{EGAII}
A.~Grothendieck and J.~Dieudonn\'{e}, \emph{El\'{e}ments de {G}\'{e}om\'{e}trie
  alg\'{e}brique {II}}, vol.~8, Inst. Hautes \'{E}tudes Sci. Publ. Math., 1961.

\bibitem{hartshornealgebraic}
R.~Hartshorne, \emph{Algebraic {G}eometry}, Springer, New York, 1977.

\bibitem{langestuhler}
H.~Lange and U.~Stuhler, \emph{Vektorb\"undel auf {K}urven und {D}arstellungen
  der {algebra\-i\-schen} {F}undamentalgruppe}, Math. Zeitschrift \textbf{156}
  (1977), 73--83.

\end{thebibliography}
\bibliographystyle{amsplain}
\end{document}